\documentclass[11pt,english]{amsart}

\usepackage[english]{babel}
\usepackage[T1]{fontenc}
\usepackage{lmodern}
\usepackage{amssymb}

\usepackage[centering]{geometry}

\theoremstyle{plain}
\newtheorem{theo}{Theorem}
\newtheorem{lemm}[theo]{Lemma}
\newtheorem{prop}[theo]{Proposition}
\newtheorem{coro}[theo]{Corollary}

\newtheorem{fact}[theo]{Fact}
\theoremstyle{definition}
\newtheorem{defi}[theo]{Definition}

\newtheorem{exem}[theo]{Example}
\theoremstyle{remark}
\newtheorem{rema}[theo]{Remark}

\usepackage{microtype}
\usepackage{enumerate}
\usepackage{graphicx}
\usepackage{color}
\usepackage{bm}
\usepackage{mathtools}
\usepackage[dvipsnames]{xcolor}
\definecolor{FlatRed}{RGB}{231,76,60}
\definecolor{FlatGreen}{RGB}{46,204,113}
\definecolor{FlatBlue}{RGB}{52,152,219}
\definecolor{FlatYellow}{RGB}{241,196,15}
\colorlet{FlatViolet}{FlatRed!50!FlatBlue}
\colorlet{FlatBrown}{FlatRed!50!FlatGreen}
\colorlet{FlatOrange}{FlatRed!50!FlatYellow}
\colorlet{FlatCyan}{FlatGreen!50!FlatBlue}

\newcommand{\R}{{\mathbb R}}

\newcommand{\N}{{\mathbb N}}

\newcommand{\D}{{\mathbb D}}
\newcommand{\T}{{\mathbb T}}
\newcommand{\RRe}{\mathrm{Re}}
\newcommand{\IIm}{\mathrm{Im}}
\newcommand{\id}{\mathrm{Id}}
\newcommand{\kernel}{\mathrm{Ker}}

\usepackage{enumitem}
\usepackage{textcomp}
\usepackage[normalem]{ulem}

\usepackage{hyperref}

\title{Random analytic functions via Gaussian multiplicative chaos}
\thanks{Y.H. is partially supported by National Key R\&D Program of China (No. 2022YFA1006300), and is grateful for the support from ERC Advanced Grant 741487 QFPROBA, as most of this work was carried out at the University of Helsinki. E.S. is supported by the Finnish Academy grant 309940.}

\author{Yichao Huang}
\address{Beijing Institute of Technology, School of Mathematics and Statistics, Beijing, China}
\email{yichao.huang@bit.edu.cn}

\author{Eero Saksman}
\address{University of Helsinki, Department of Mathematics and Statistics, P.O. Box 68, FIN-00014 University of Helsinki, Finland}
\email{eero.saksman@helsinki.fi}

\begin{document}

\begin{abstract}
We define a random analytic function $\varphi$ on the unit disc by letting a Gaussian multiplicative measure to be one of its Clark measures. We show that $\varphi$ is almost surely a Blaschke product and we provide rather sharp estimates for the density of its zeroes.
\end{abstract}

\dedicatory{The authors dedicate the paper in honor of Professor H\aa kan Hedenmalm on his 60th birthday$^\ast$}\thanks{$^\ast$We have to admit that there is a little delay from our part\dots}

\maketitle

\section{Introduction}
Given an analytic self map $\varphi:\D\to\D$ of the unit disc in the complex plane, for each $\alpha\in\T\coloneqq\partial\D$ the measure $\nu_\alpha\coloneqq\nu_{\varphi,\alpha}$ is defined via
\begin{equation}\label{eq:clark}
\RRe\left(\frac{\alpha+\varphi(z)}{\alpha-\varphi(z)}\right)\;=\;\int_{\T}\frac{1-|z|^2}{|e^{i\theta}-z|^2}\nu_\alpha(d\theta),\qquad z\in\D.
\end{equation}
The measure $\nu_\alpha$, or especially its singular part, describes how strongly and where on the boundary the function $\varphi$ takes the value $\alpha$ -- note that this can happen only at the boundary.  The \emph{Clark measures} $\nu_\alpha$ thus obtained for $\alpha\in\T$  have been much studied especially in connection with applications to perturbative operator theory, including the spectral theory of $1$-dimensional Schr\"odinger equation. Clark measures are defined in \cite{clark1972one} and further studied by Aleksandrov \cite{zbMATH04057865,aleksandrov1989inner,aleksandrov1994maximum,aleksandrov1995existence,aleksandrov1996isometric}, so they are also called Aleksandrov measures. We refer to the reviews \cite{poltoratski2006} and \cite{Saksman2007AnEI} for the basic properties of this interesting family of measures.

In the present paper we will employ the Clark measures to construct random analytic functions on the unit disc by defining  $\varphi:\D\to \D$  as the analytic self map of $\D$ such that its Clark measure at point $\alpha=1$ equals a given random measure $\nu$. Especially, we will consider the choice
\begin{equation}\label{eq:choice}
 \nu_{\varphi,1}=\mu_\gamma\coloneqq\text{``}\exp(\gamma X(\theta))\text{''}, \qquad \gamma\in (0,\sqrt{2}].
\end{equation}
Here  $\gamma\in (0,\sqrt{2}]$ is a fixed parameter, and  $\text{``}\exp(\gamma X(\theta))\text{''}$ stands for the (random) \emph{Gaussian multiplicative chaos measure} corresponding to a log-correlated Gaussian field $X$. The most natural choice for $X$ on $\T$ is the so-called ``canonical log-correlated field'' 
\begin{equation*}
X_c(\theta)=\sum_{n=1}^\infty n^{-1/2}\big(A_n\cos(n\theta) + B_n\sin(n\theta)\big), \qquad \theta\in[0,2\pi),
\end{equation*}
where the $A_n,B_n$ are independent, identically distributed standard Gaussians. In general, log-correlated Gaussian fields $X$ on $\T$ have the covariance structure
\begin{equation}\label{eq:cov}
\mathbb{E}\left[X(\theta_1)X(\theta_2)\right]=\log\frac{1}{|e^{i\theta_1}-e^{i\theta_2}|}+g(\theta_1,\theta_2)
\end{equation}
where $g:\T^2\to\R$ is a (symmetric) continuous function. For the canonical log-correlated field $X_c$, one has $g\equiv 0$.

Multiplicative chaos measures form a natural and important class of random positive measures that has re-appeared during last $15$ years in various important roles in statistical physics and other applications. We recall their proper definition and basic properties in Section~\ref{sec:chaos}, and we also recall  the definition of critical chaos measures at the threshold $\gamma=\sqrt{2}$ when it is needed. Overall, we refer the reader to the review \cite{Rhodes_2014} for a general background.

It is known that the mesures $\mu_\gamma$ for $\gamma\in(0,\sqrt{2})$ and $\gamma=\sqrt{2}$ are almost surely purely singular with respect to the Lebesgue measure on $\T$. The basic theory of Clark measures then implies that $\varphi$ is almost surely an inner function, i.e. $|\varphi(e^{i\theta})|=1$ for almost every $\theta$, see Proposition~\ref{Prop:Inner}. Any inner function $\varphi$ admits the decomposition (see \cite[Theorem~5.5]{garnett2007bounded})
\begin{equation*}
\varphi(z)=B(z)S(z),
\end{equation*}
where $B$ is a \emph{Blaschke product}:
\begin{equation*}
B(z)=e^{ic}z^{m}\prod\limits_{|z_n|\neq 0}\frac{\overline{z}_n}{|z_n|}\frac{z_n-z}{1-\overline{z}_n z}
\end{equation*}
with $c\in\R$, $m\geq 0$, $z_n\in\D$ and $\sum_{n=1}^\infty(1-|z_n|)<\infty$, and  $S$ is a \emph{singular inner function}, i.e.
\begin{equation*}
S(z)=\exp\left(-\int_{\T}\frac{e^{i\theta}+z}{e^{i\theta}-z}d\eta(\theta)\right)
\end{equation*}
where the measure $d\eta$ is positive and singular to $d\theta$.

The harmless convention $\IIm(\varphi(0))=0$ will be used in this paper. Since we focus on the case $\alpha=1$, by the Herglotz representation formula,
\begin{equation}\label{eq:Clark}
\varphi(z)=\frac{h(z)-1}{h(z)+1},\qquad h(z)\coloneqq\int_{\T}\frac{e^{i\theta}+z}{e^{i\theta}-z}\nu_1(d\theta).
\end{equation}

Our first result shows that the singular part of our random function $\varphi$ is almost surely trivial:
\begin{theo}\label{thm1} Let $\gamma\in (0,\sqrt{2}]$ (where $\gamma=\sqrt{2}$ refers to the critical chaos) and let $\mu_\gamma$ be the chaos measure on $\T$ corresponding to a log-correlated field $X$ with the covariance structure \eqref{eq:cov}. Assume that $g\in W^{s,2}(\T^2)$ for some $s>1$. Then, the inner function $\varphi$ defined via \eqref{eq:Clark} with the choice $\nu_1=\mu_\gamma$ is almost surely a Blaschke product.
\end{theo}

This result can be thought of as a random analogue of a famous result of Frostman \cite[Theorem~6.4]{garnett2007bounded}, which states that given an inner function $f$, for almost every $a\in\D$ the ``Frostman shift'' $\tau_a\circ f$ is a pure Blaschke product. Here $\tau_a$ is the M\"obius automorphism $\tau_a:\D\to\D$, and $\tau_a(z)=\frac{a-z}{1-\overline{a}z}.$

We also provide an estimate for the density of the zeroes of the random Blaschke product $\varphi$:
\begin{theo}\label{thm2}
Assume that $g\in W^{2,2}(\T^2)$ in~\eqref{eq:cov} and let $z_1,z_2,\dots$ stand for the zeroes of $\varphi$. Then we have the following phase transition for the density of zeroes:
\begin{enumerate}
    \item Almost surely, $\sum_{k=1}^{\infty} (1-|z_k|)^\beta<\infty$ for $1-\gamma^2/8<\beta$.
    \item Almost surely, $\sum_{k=1}^{\infty} (1-|z_k|)^{\beta}=\infty$ for $1-\gamma^2/8>\beta$.
\end{enumerate}
\end{theo}
The above result is slightly surprising in the sense that the density of the zeroes goes up when the parameter $\gamma$ of the chaos measure decreases. On the other hand, it is natural when one notes that in a heuristic way the support of $\mu_\gamma$ ``decreases'' in a sense when $\gamma$ increases. One may ask what happens in the supercritical case, especially how does the analogue of Theorem \ref{thm2} look like for $\gamma >\sqrt{2}$. One of the  questions that we plan to study in the future is how the zero set behaves realization wise as the parameter $\gamma$ varies.

The present note is a part of our long-term project to study some aspects of random spectral theory via Clark measures. We thank Alexei Poltoratski for reawakening our interest in these questions, and especially we are grateful to discussions with H\aa kan Hedenmalm who independently suggested us (during a visit to the University of Helsinki a couple of years ago) that one should expect a result to the direction of Theorem~\ref{thm1} above.

\section{Log-correlated Gaussian fields and Gaussian multiplicative chaos}\label{sec:chaos}
In the following, we consider several random (generalized) functions defined on the unit disc $\mathbb{D}\subset\mathbb{R}^2$, on the unit circle $\T=\partial\mathbb{D}$, or on an interval of the real axis, say $\left[-\frac{1}{2},\frac{1}{2}\right]\subset\mathbb{R}$. We will usully denote by $X,Y,Z,\dots$ some log-correlated Gaussian fields. The parameter $\gamma\in\left(0,\sqrt{2}\right)$ will be fixed and $\mu\coloneqq\mu_\gamma$ will denote the (subcritical) Gaussian multiplicative chaos measure with parameter $\gamma$. The critical case $\gamma=\sqrt{2}$ will be recalled separately in the proofs.

Since this note is intended for readers from both probability and analysis backgrounds, we recall some definitions and basic facts about log-correlated Gaussian fields and Gaussian multiplicative chaos.

\subsection{Log-correlated Gaussian fields on the unit circle}
Consider the Gaussian field $X_c$ on the unit circle $\T$ defined by a random Fourier series:
\begin{equation}\label{eq:1}
X_c(\theta)=\sum\limits_{n=1}^{\infty}n^{-\frac{1}{2}} (A_n\cos(n\theta)+B_n\sin(n\theta)), \qquad \theta\in [0,2\pi), 
\end{equation}
where $A_n, B_n$ and i.i.d. standard Gaussian variables. We refer to this particular log-correlated field as the \emph{canonical field} (on $\T$). Let us recall why the  Gaussian field $X_c(\theta)$ is log-correlated -- actually its covariance kernel writes in the simple form
\begin{equation}\label{eq:canonicalcovariance}
C(\theta,\theta')\coloneqq\mathbb{E}\left[X(\theta)X(\theta')\right]=-\log\left|e^{i\theta}-e^{i\theta'}\right|.
\end{equation}
Indeed, a  direct calculation yields
\begin{equation*}
C(\theta,\theta')=\sum\limits_{n=1}^{\infty}\frac{1}{n}\cos(n(\theta-\theta')).
\end{equation*}
To check that this is the same as $-\log|e^{i\theta}-e^{i\theta'}|$, write
\begin{equation*}
-\log\left|e^{i\theta}-e^{i\theta'}\right|=-\log\left|1-e^{i(\theta-\theta')}\right|=-\frac{1}{2}\log\left(1-e^{i(\theta-\theta')}\right)-\frac{1}{2}\log\left(1-e^{-i(\theta-\theta')}\right)
\end{equation*}
and develop the Taylor series.

Due to the divergence of $-\log\left|e^{i\theta}-e^{i\theta'}\right|$ on the diagonal $\theta=\theta'$, the Gaussian field $X_c$ cannot be defined as a random function, but only as a random distribution in the sense of Schwartz. Indeed, one verifies that the field $X_c$ almost surely lives in the negative order 
Sobolev space $W^{-s,2}$ for all $s>0$. Recall that the Sobolev spaces $W^{s,2}$ for $s\in \R$ are defined as
\begin{equation*}
    W^{s,2}(\T)\coloneqq\left\{f\in L^2(\T)~;~||f||_{s,2}^{2}\coloneqq\sum\limits_{n=-\infty}^{\infty}\left(1+|n|^2\right)^{s}|\hat{f}(n)|^2<\infty\right\}.
\end{equation*}
Then $W^{-s,2}$ and $W^{s,2}$ are mutual duals under the natual pairing. For relevant properties of these classical function spaces $W^{s,2}$ with $s\in\mathbb{R}$ in relation to Gaussian multiplicative chaos, we refer to \cite[Section~2.2]{junnila2019decompositions}.

We will consider more general log-correlated covariance kernels of the form
\begin{equation*}
    K(\theta,\theta')=C(\theta,\theta')+g(\theta,\theta')
\end{equation*}
with $g\in W^{s,2}(\T^2)$ for some $s>1$. The Gaussian field $X$ on $\T$ with kernel $K$ is also called a log-correlated Gaussian field, and the above regularity remark applies also to $X$.

\subsection{Gaussian multiplicative chaos on the unit circle}
Fix $\gamma\in(0,\sqrt{2})$. The theory of Gaussian multiplicative measures developed by Kahane \cite{kahane1985chaos} as well as its recent developments take care of rigorously exponentiating a log-correlated Gaussian field such as $\gamma X$.

\begin{defi}[Gaussian multiplicative chaos measures]\label{prop:GMCdefinition}
Let $Y$ denote any log-correlated Gaussian field on $\T$ with covariance kernel
\begin{equation*}
K_Y(e^{i\theta},e^{i\theta'})=-\log\left|e^{i\theta}-e^{i\theta'}\right|+g\left(e^{i\theta},e^{i\theta'}\right)
\end{equation*}
where $g$ is a (symmetric) continuous function on $\T\times\T$.

Let $Y_{\varepsilon}$ be a standard $\varepsilon$-mollification  of $Y$ using a compact and smooth test function.  The Gaussian multiplicative chaos measure $\mu_Y$ on $\T$, associated to $Y$ with parameter $\gamma\in\left(0,\sqrt{2}\right)$, is defined as the following limit taken in probability:
\begin{equation*}
d\mu_Y(\theta)\coloneqq \lim\limits_{\varepsilon\to 0}e^{\gamma Y_\varepsilon(e^{i\theta})-\frac{\gamma^2}{2}\mathbb{E}\left[Y_\varepsilon(e^{i\theta})^2\right]}d\theta.
\end{equation*}
Here the convergence of the measures is  in the sense of weak$^*$ convergence, i.e. for any continuous test function $\phi:\T\to\mathbb{R}$ one has 
\begin{equation*}
\int_{\T}\phi(e^{i\theta})d\mu_Y(\theta)=\lim\limits_{\varepsilon\to 0}\int_{\T}e^{\gamma Y_\varepsilon(e^{i\theta})-\frac{\gamma^2}{2}\mathbb{E}\left[Y_\varepsilon(e^{i\theta})^2\right]}\phi(e^{i\theta})d\theta.
\end{equation*}

The Gaussian multiplicative chaos measure $\mu_Y$ is unique (i.e. does not depend on the mollification) and non-trivial for all $\gamma^2<2$ (we refer to this as the subcritical regime) and degenerate (i.e. almost surely zero everywhere) if $\gamma^2\geq 2$.
\end{defi}
The convergence in law or in probability for  Gaussian multiplicative chaos measures (i.e. that the above definition makes sense) was established  essentially in \cite{kahane1985chaos}, and later on the theory has been developed in many works, e.g. \cite{bacry2003log, robert2010gaussian, SHAMOV20163224, Berestycki_2017}.

\begin{rema}[$1d$-Gaussian multiplicative chaos measures]
The more classical setting, completely similar to the above proposition, defines the Gaussian multiplicative chaos measures on a compact subset of the real line $\mathbb{R}$. In this note, we will use mostly the interval $\left[-\frac{1}{2},\frac{1}{2}\right]\in\mathbb{R}$. Both of them are $1d$-Gaussian multiplicative chaos measures, as are the chaos measures defined on $\T$.
\end{rema}

We first record some important known properties of the Gaussian multiplicative chaos measures in one-dimension, and we will provide references after stating the results.
\begin{fact}[Existence of moments]\label{fact:Positivemoment}
Let $Y$ be a log-correlated Gaussian field on $\T$ or $\left[-\frac{1}{2},\frac{1}{2}\right]$ and $\mu_Y$ the associated Gaussian multiplicative chaos measure with $\gamma\in\left(0,\sqrt{2}\right)$ as in Definition~\ref{prop:GMCdefinition}. Then for all sets $T\subset \T$ or $T\subset\left[-\frac{1}{2},\frac{1}{2}\right]$ with non-empty interior,
\begin{equation*}
\mathbb{E}\left[\mu_Y(T)^{p}\right]<\infty
\end{equation*}
if and only if $p<\frac{2}{\gamma^2}$ (this includes finiteness of all  negative moments).
\end{fact}
Especially one may note that in the subcritical regime $\gamma^2<2$, the first moment of the mass of a $1d$-Gaussian multiplicative chaos measure always exists.

\begin{fact}[Exact scaling log-correlated fields]\label{fact:ScalingPurelog}
Consider the log-correlated Gaussian field $Z$ on the interval $\left[-\frac{1}{2},\frac{1}{2}\right]$ with covariance kernel:
\begin{equation*}
K_Z(\theta,\theta')=\log\frac{1}{|\theta-\theta'|}.
\end{equation*}
Then $Z$ is (locally) translation invariant and exact scale invariant: in particular, for all $0<r<1$, the following Gaussian fields are equal in law:
\begin{equation*}
\left(Z(\theta)+\sqrt{-\log r}N\right)_{\theta\in\left[-\frac{1}{2},\frac{1}{2}\right]}\qquad\text{and}\qquad\left(Z\left(r^{-1}\theta\right)\right)_{\theta\in\left[-\frac{r}{2},\frac{r}{2}\right]}.
\end{equation*}
where $N$ is a standard Gaussian independent of the field $Z$.

As a consequence, the following equality in law holds for the Gaussian multiplicative chaos measure $\mu_{Z}$ associated to the field $Z$ with parameter $\gamma\in\left(0,\sqrt{2}\right)$ and all $[x,x+r]\subset \left[-\frac{1}{2},\frac{1}{2}\right]$:
\begin{equation*}
\int_{[x,x+r]}d\mu_{Z}(\theta)\sim r^{1+\gamma^2/2}e^{\gamma \sqrt{-\log r}N}\int_{\left[-\frac{1}{2},\frac{1}{2}\right]}d\mu_Z(\theta).
\end{equation*}
We refer to the last identity as ``exact scaling'' relation in the sequel.
\end{fact}

\begin{fact}[Kahane's convexity inequality]\label{fact:ScalingKahane}
Let $Y_1$, $Y_2$ be two continuous centered Gaussian processes indexed by $\T$ or on any interval $I\subset\R$, such that for all $\theta,\theta'$,
\begin{equation*}
\mathbb{E}\left[Y_1(\theta)Y_1(\theta')\right]\leq \mathbb{E}\left[Y_2(\theta)Y_2(\theta')\right].
\end{equation*}

Then, for all convex function $F$ with at most polynomial growth at infinity and any positive finite measure $\sigma$ on $\T$ or on $I$,
\begin{equation*}
\mathbb{E}\left[F\left(\int e^{Y_1(\theta)-\mathbb{E}[Y_1(\theta)^2]}d\sigma(\theta)\right)\right]\leq \mathbb{E}\left[F\left(\int e^{Y_2(\theta)-\mathbb{E}[Y_2(\theta)^2]}d\sigma(\theta)\right)\right].
\end{equation*}
\end{fact}

\begin{proof}[Proofs of the above properties]
See respectively \cite[Theorem~2.5]{Rhodes_2014}, \cite[Theorem~2.11]{Rhodes_2014}, \cite[Theorem~2.16]{Rhodes_2014} and \cite[Theorem~2.1]{Rhodes_2014}. To apply the last convexity inequality to log-correlated Gaussian fields, we systematically use the regularization procedure in Proposition~\ref{prop:GMCdefinition} to meet the continuity assumption: this is standard and we will not write out the regularization everytime.
\end{proof}

\begin{rema}[Comparison of moments]\label{re:comp}
For people with background in analysis we now demonstrate a standard use of Kahane's convexity inequality mainly  for comparing moments of two different chaos measures, since we need to do this later on with careful book-keeping of the constants involved.  Consider thus $p<\frac{2}{\gamma^2}$ and the function $F(x)=x^{p}$ for $x>0$. Then either $F$ or $-F$ is convex. Consider a general log-correlated Gaussian field $Y$ as in Definition~\ref{prop:GMCdefinition} and an exact scaling log-correlated Gaussian field $Z$ as in Fact~\ref{fact:ScalingPurelog}. Since the correction term $g(\theta,\theta')$ in the covariance function of the field $Y$ is bounded in absolute value by a finite constant $C$, we can apply Kahane's inequality to the following Gaussian fields (with $N$ a normal Gaussian variable independent of $Y$ and $Z$):
\begin{equation*}
\left(Y(\theta)+\sqrt{C}N\right)_{\theta} \qquad\text{and}\qquad \left(Z(\theta)\right)_{\theta},
\end{equation*}
or 
\begin{equation*}
Y(\theta) \qquad\text{and}\qquad \left(Z(\theta) + \sqrt{C}N\right)_{\theta},
\end{equation*}
to conclude that there exists some positive constant $K$ such that, for all positive functions $P\left(e^{i\theta}\right)$,
\begin{equation*}
\frac{1}{K}\mathbb{E}\left[\left(\int_{\T}P(e^{i\theta})d\mu_{Y}(\theta)\right)^{p}\right]\leq \mathbb{E}\left[\left(\int_{\T}P(e^{i\theta})d\mu_{Z}(\theta)\right)^{p}\right]\leq K\mathbb{E}\left[\left(\int_{\T}P(e^{i\theta})d\mu_{Y}(\theta)\right)^{p}\right].
\end{equation*}
This relation tells us that modulo a multiplicative constant, the moments of the positively weighted mass of a general Gaussian multiplicative chaos measure associated to any log-correlated Gaussian field $Y$ scales similarly like in the exact scaling case with the field $Z$. We will write the above equation as
\begin{equation*}
\mathbb{E}\left[\left(\int_{\T}P(e^{i\theta})d\mu_{Z}(\theta)\right)^{p}\right]\asymp \mathbb{E}\left[\left(\int_{\T}P(e^{i\theta})d\mu_{Y}(\theta)\right)^{p}\right].
\end{equation*}
The constant in $\asymp$ only depends on the parameter $\gamma$, the correction term $g\left(e^{i\theta},e^{i\theta'}\right)$, and $p$.

To summarize, when estimating the moment of a positively weighted mass of a Gaussian multiplicative chaos measure $\mu_{Y}$, we can switch the underlying Gaussian field to an exact scaling log-correlated Gaussian field $Z$ of Fact~\ref{fact:ScalingPurelog}, up to a deterministic multiplicative constant.
\end{rema}

When $\mu=\mu_{X_c}$ is the Gaussian multiplicative chaos measure on $\T$ associated to $X$ with the canonical covariance kernel in \eqref{eq:canonicalcovariance}, the measure $d\mu(\theta)$ is then invariant under rotations of the unit circle. In fact, with this particular choice, the measure $\mu$ transforms nicely under any M\"obius transformation of the unit disc $\mathbb{D}$. Since we don't need this fact as we study general log-correlated Gaussian fields, we leave the interested reader to consult~\cite{vargas2017lecture} for the fascinating connections between Gaussian multiplicative chaos measures and Conformal Field Theory.

We record a simple auxiliary result.
\begin{lemm}[Negative moments]\label{lem:negativemoment}
Let $I_1,I_2$ be two  open intervals of the unit circle $\T$ or the interval $\left[-\frac{1}{2},\frac{1}{2}\right]$ and let $\mu$ be any Gaussian multiplicative chaos measure on $\T$ or $\left[-\frac{1}{2},\frac{1}{2}\right]$ with parameter $\gamma\in\left(0,\sqrt{2}\right)$. Then the almost surely positive random variable
\begin{equation*}
\inf\{\mu(I_1),\mu(I_2)\}
\end{equation*}
has negative moments of any order.
\end{lemm}
\begin{proof}
For $\gamma\in(0,\sqrt{2}]$, any $1d$-Gaussian multiplicative chaos measure $\mu$ has all orders of negative moments on any open interval by Fact~\ref{fact:Positivemoment}. The lemma follows by using
\begin{equation*}
\left(\inf\{a,b\}\right)^{-p}\leq a^{-p}+b^{-p}
\end{equation*}
for all $a,b>0$ and taking expectations.
\end{proof}

\begin{lemm}[Structure of the carrier]\label{lem:atomlessness}
For $\gamma\in\left(0,\sqrt{2}\right)$, any Gaussian multiplicative chaos measure on $\T$ is almost surely singular with respect to the Lebesgue measure and does not possess atoms.
\end{lemm}
\begin{proof}
See \cite[Theorem~2.6]{Rhodes_2014} for a proof and furthur discussions.
\end{proof}

It follows immediately that the analytic self-map $\varphi:\D\to\D$ (defined via \eqref{eq:choice} and \eqref{eq:clark})  is a.s. an inner function, as alluded in the introduction.
\begin{prop}[Inner function]\label{Prop:Inner}
The random analytic function $\varphi$ is almost surely an inner function, i.e. for all $z\in\T$, we have $|\varphi(z)|=1$.
\end{prop}
\begin{proof}
It is a classical fact \cite[Theorem~2.2]{Saksman2007AnEI} on Clark measures that $\varphi$ is an inner function if and only if the Clark measure $\nu_{\varphi,1}$ is singular with respect to the Lebesgue measure. The claim follows from Lemma~\ref{lem:atomlessness}.
\end{proof}

\section{Blaschke product property for the canonical field}
We find it instrumental to first provide a short and simple proof of the Blaschke product property (see Theorem~\ref{thm1}) in the case of the canonical log-correlated Gaussian field $X_c$. In addition,  it prepares the reader gently for the perturbation philosophy applied also in  the more involved proofs in the sequel.

\subsection{A Frostman type lemma}
To prove the Blaschke product property, we first provide a variant of the classical Frostman lemma \cite[Theorem~II.6.4]{garnett2007bounded} adapted to our situation.
\begin{lemm}[A Frostman type lemma]\label{lem:Frostman}
Let  $\varphi$ be a random inner function\footnote{Naturally we assume that the map $(z,\omega)\mapsto \varphi (z)$ is jointly measurable, where $\omega$ stands for an element in the probability space.} on the unit disc. Then $\varphi$ is almost surely a Blaschke product if for some $\varepsilon>0$
\begin{equation}\label{eq:Frostman'}
\sup\limits_{z\in\mathbb{D}}\mathbb{E}\left(\log\frac{1}{|\varphi(z)|}\right)^{1+\varepsilon}<\infty.
\end{equation}
\end{lemm}
\begin{proof}
The inner function $\varphi$ is a Blaschke product if and only if (see \cite[Theorem~II.2.4, Corollary~I.6.6]{garnett2007bounded})
\begin{equation*}
\liminf\limits_{r\to 1^{-}}\int_{0}^{2\pi}\log\left(\frac{1}{|\varphi(re^{i\theta})|}\right)d\theta=0.
\end{equation*}
Define $F:\D\to [0,\infty]$ by setting $F(z)\coloneqq\mathbb{E} \log \left(\frac{1}{|\varphi(z)|} \right)$. Since almost surely $\varphi$ is an inner function, and hence a.s. its boundary values are unimodular a.e. at the boundary $\partial\D$, we deduce by Fubini that for a.e. $\theta\in [0,2\pi]$ it holds that a.s. $\lim_{r\to 1} \log(1/|\varphi(re^{i\theta})|)=0$. By our assumption, we may apply uniform integrability for the expectation in the definition  of $F(re^{i\theta})$ and deduce that
\begin{equation*}
\lim_{r\to 1^-} F(re^{i\theta})=0 \qquad \textrm{for almost every} \qquad \theta\in [0,2\pi).
\end{equation*}

Thereafter (note that $F$ is bounded under our assumption) we may apply dominated convergence theorem to deduce that  $\lim_{r\to 1^-} \int_0^{2\pi} F(re^{i\theta})d\theta =0$, whereafter Fubini's theorem  and Fatou's lemma yield that
\begin{equation*}
\mathbb{E}\left[\liminf\limits_{r\to 1^{-}} \int_{0}^{2\pi}\log\left(\frac{1}{|\varphi(re^{i\theta})|}\right)d\theta\right]
\leq \liminf\limits_{r\to 1^{-}}  \int_0^{2\pi} F(re^{i\theta})d\theta =0.
\end{equation*}
This finishes the proof.
\end{proof}

Since $\RRe(h(z))\geq 0$, the quantity inside the expectation only blows up when $h(z)$ is close to $1$. It is then an elementary exercise to verify that the following condition suffices:
\begin{coro}\label{cor:Frostman}
Let $\varphi$ be defined via  Equation~\eqref{eq:Clark} with a random singular measure $\mu$. Then $\varphi$ is almost surely a Blaschke product if for some $\varepsilon>0$
\begin{equation*}
\sup\limits_{z\in\mathbb{D}}\mathbb{E}\left(\log^+\frac{1}{|h(z)-1|}\right)^{1+\varepsilon}<\infty.
\end{equation*}
\end{coro}

\begin{rema} Since in the proof of Lemma \ref{lem:Frostman} we used the assumption only to guarantee uniform integrability under the expectation,  and at the same time boundedness of $F$, we could equally well replace the condition
\eqref{eq:Frostman'} by 
\begin{equation*}
\sup\limits_{z\in\mathbb{D}}\mathbb{E}\; \psi\left(\log\frac{1}{|\varphi(z)|}\right)<\infty,
\end{equation*}
where $\psi:[0,\infty)\to[0,\infty)$ is any increasing function with $\lim_{x\to\infty} \psi(x)/x=\infty.$
\end{rema}

\subsection{Proof of the Blaschke product property}
We now prove Theorem~\ref{thm1} for the canonical log-correlated Gaussian field $X_c$. Since the measure $\mu$ is invariant under rotations of $\T$, it is enough to show that
\begin{equation*}
\sup\limits_{r\in[0,1)}\mathbb{E}\left(\log^{+}\frac{1}{|h(r)-1|}\right)^{1+\varepsilon}<\infty.
\end{equation*}

Since for any $p\in(0,1)$ and  $\varepsilon >0$ we have
\begin{equation*}
\left(\log^{+}\frac{1}{x}\right)^{1+\varepsilon}\leq \frac{C(p,\varepsilon)}{x^p}\qquad\text{and}\qquad |h(r)-1|\geq|\IIm(h(r))|,
\end{equation*}
it is sufficient to consider the imaginary part $\IIm(h(r))$ and show that for some $p\in(0,1)$,
\begin{equation*}
\sup\limits_{r\in[0,1)}\mathbb{E}\left[\frac{1}{|\IIm(h(r))|^p}\right]<\infty.
\end{equation*}
Notice that the imaginary part $\IIm(h(r))$ writes as a signed Gaussian multiplicative chaos measure:
\begin{equation*}
\IIm(h(r))=\int_{\T}\frac{-2r\sin(\theta)}{|r-e^{i\theta}|^2}d\mu(\theta).
\end{equation*}

The key idea is to take out a Fourier mode in the canonical field $X_c$ and write the following independent decomposition of $X_c$:
\begin{equation*}
X_c(\theta)=B_1\sin(\theta)+\widetilde{X}(\theta).
\end{equation*}
It follows that (where $\widetilde{\mu}$ is the Gaussian multiplicative measure associated with the log-correlated Gaussian field $\widetilde{X}$):
\begin{equation*}
d\mu(\theta)=\exp\left(\gamma B_1\sin(\theta)-\frac{\gamma^2\sin^2(\theta)}{2}\right)d\widetilde{\mu}(\theta),
\end{equation*}
and the following derivative has constant sign:
\begin{equation*}
-\frac{\partial \IIm(h(r))}{\partial B_1}=\int_{\T}\frac{2\gamma r\cdot \sin^2(\theta)}{|r-e^{i\theta}|^2}\cdot\exp\left(\gamma B_1\sin(\theta)-\frac{\gamma^2\sin^2(\theta)}{2}\right)d\widetilde{\mu}(\theta)\geq 0.
\end{equation*}

Now if $B_1\geq 0$, by looking at the interval where $\sin(\theta)\geq 0$, the integral is bigger than
\begin{equation*}
e^{-\gamma^2/2}\int_{\theta\in[0,\pi]}\frac{2\gamma r\cdot\sin^2(\theta)}{|r-e^{i\theta}|^2}d\widetilde{\mu}(\theta).
\end{equation*}
Similarly if $B_1\leq 0$, by looking at the interval where $\sin(\theta)\leq 0$, the integral is bigger than
\begin{equation*}
e^{-\gamma^2/2}\int_{\theta\in[\pi,2\pi]}\frac{2\gamma r\cdot\sin^2(\theta)}{|r-e^{i\theta}|^2}d\widetilde{\mu}(\theta).
\end{equation*}
In any case, for all $B_1\in\mathbb{R}$ and $r\in [1/2,1)$,
\begin{equation*}
-\frac{\partial \IIm(h(r))}{\partial B_1}\geq C\inf\left\{\int_{\theta\in[\frac{\pi}{4},\frac{3\pi}{4}]}d\widetilde{\mu}(\theta),\int_{\theta\in[\frac{5\pi}{4},\frac{7\pi}{4}]}d\widetilde{\mu}(\theta)\right\}.
\end{equation*}

The above is a lower bound on the derivative of $\IIm(h(r))$ with respect to the Gaussian variable $B_1$ using the auxiliary field $\widetilde{X}$. We now record an elementary observation.
\begin{lemm}\label{lem:integrallemma}
Let $\frac{1}{\sqrt{2\pi}}e^{-\frac{\rho^2}{2}}d\rho$ denote the distribution density of a standard Gaussian random variable. Let $v:\mathbb{R}\to\mathbb{R}$ be a differentiable function such that $v'(\rho)\geq b>0$ for all $\rho\in\mathbb{R}$. Then
\begin{equation*}
\frac{1}{\sqrt{2\pi}}\int_{\mathbb{R}}e^{-\frac{\rho^2}{2}}\frac{1}{|v(\rho)|^p}d\rho\leq Cb^{-p},
\end{equation*}
with $C=C(p)$, for any $p\in (0,1)$.
\end{lemm}
\begin{proof}
Let $x_0\in\mathbb{R}$ be the only point where $v(x_0)=0$. Then
\begin{equation*}
|v(\rho)-v(x_0)|\geq b|\rho-x_0|
\end{equation*}
and
\begin{equation*}
\int_{\mathbb{R}}e^{-\frac{\rho^2}{2}}\frac{1}{|v(\rho)-v(x_0)|^p}d\rho\leq b^{-p}\left(\int_{\mathbb{R}}e^{-\frac{\rho^2}{2}}\frac{1}{|\rho-x_0|^{p}}d\rho\right).
\end{equation*}
It remains to see that
\begin{equation*}
\int_\R e^{-\frac{\rho^2}{2}}\frac{1}{|\rho-x_0|^{p}}d\rho\leq\int_{x_0-1}^{x_0+1}\frac{1}{|\rho-x_0|^{p}}d\rho+\int_{-\infty}^{x_0-1}e^{-\frac{\rho^2}{2}}d\rho+\int_{x_0+1}^{\infty}e^{-\frac{\rho^2}{2}}d\rho\eqqcolon C(p)
\end{equation*}
which is finite as long as $p\in(0,1)$.
\end{proof}

To conclude, let $\rho$ parametrize the Gaussian variable $B_1$ and let $b$ be the above lower bound
\begin{equation*}
    b=C\inf\left\{\int_{\theta\in[\frac{\pi}{4},\frac{3\pi}{4}]}d\widetilde{\mu}(\theta),\int_{\theta\in[\frac{5\pi}{4},\frac{7\pi}{4}]}d\widetilde{\mu}(\theta)\right\}
\end{equation*}
independent of $B_1$, and apply Lemma~\ref{lem:integrallemma}. Taking the expectation with respect to $B_1$ is equivalent to integrating over $\rho$, and we get the following for the conditional expectation on $B_1$ with $p\in(0,1)$:
\begin{equation*}
\mathbb{E}\left[|\IIm(h(r))|^{-p}\,\big|\, \widetilde{X}\right]\leq C\left(\inf\left\{\int_{\theta\in[\frac{\pi}{4},\frac{3\pi}{4}]}d\widetilde{\mu}(\theta),\int_{\theta\in[\frac{5\pi}{4},\frac{7\pi}{4}]}d\widetilde{\mu}(\theta)\right\}\right)^{-p}.
\end{equation*}
Now taking the expectation with respect to all other Gaussian variables,
\begin{equation*}
\mathbb{E}\left[\frac{1}{|\IIm(h(r))|^p}\right]\leq C\mathbb{E}\left[\left(\inf\left\{\int_{\theta\in[\frac{\pi}{4},\frac{3\pi}{4}]}d\widetilde{\mu}(\theta),\int_{\theta\in[\frac{5\pi}{4},\frac{7\pi}{4}]}d\widetilde{\mu}(\theta)\right\}\right)^{-p}\right].
\end{equation*}
The last quantity, independent of $r$, is finite for all $p\in(0,1)$ by Lemma~\ref{lem:negativemoment}: this finishes the proof.

\section{On the density of zeroes: preparations}\label{sec:preparations}
In this section, we prepare our study for the density problem, Theorem~\ref{thm2}. Section \ref{subse:Frostman}  first develops a quantitative version of the Frostman-type Lemma~\ref{lem:Frostman}, and provides a sufficient condition to Theorem~\ref{thm2}. It  turns out that we again need  uniform estimate for the expectation of a suitable functional with a log-singularity, acting on the real part of the Poisson extension of the measure. 

In Section \ref{subse:control}  we then get rid of the log-singularity using a perturbative method, whose idea is similar to the one illustrated in the previous section, although the situation here is much more complicated since we need to extract two independent perturbations with suitable properties. Their existence is nontrivial since we are dealing with a general log-correlated field, and the technical details are postponed to Section \ref{subse:existence}. The upshot of Section \ref{sec:preparations} is Corollary \ref{coro:FinalSufficient} below,  which yields  a sufficient condition to the first part of Theorem~\ref{thm2} in the form of a moment estimate of  a weighted auxiliary positive Gaussian multiplicative chaos measure.

\subsection{Lemma on the density of the zeroes}\label{subse:Frostman}
The goal here is to provide a simple quantitative version of Lemma~\ref{lem:Frostman}. In the sequel, the symbol $\lesssim$ means smaller or equal to within a multiplicative constant independent of $r=|z|$.

\begin{lemm}\label{lemm:densityupperbound}
Let $z_1,z_2,\dots$ stand for the zeroes of a Blaschke product $f:\D\to\D$ and let $0<\beta<1$. We have 
\begin{equation*}
    \sum_{k=1}^{\infty} (1-|z_k|)^\beta<\infty
\end{equation*}
if and only if $\displaystyle \int_{\mathbb{D}}(1-|z|^2)^{\beta-2}\log(1/|\varphi(z)|)dA(z)<\infty$.
\end{lemm}
\begin{proof}
Assume first that $\varphi$ is a finite Blaschke product. We set $v(z)=(a-|z|^2)^{\beta}$  with $0<\beta <1$ and first let the auxiliary parameter $a$ satisfy $a>1$. Compute
\begin{equation*}
\Delta v(z)= 4\beta\big(\beta|z|^2-a\big)(a-|z|^2)^{\beta-2}\asymp (a-|z|^2)^{\beta-2}\quad\textrm{for}\;\, z\in\D,
\end{equation*}
where the implicit constants in $\asymp$ are independent of $a>1$.  We note that $v$ is smooth on $\overline\D$ up to the boundary, and so is $z\mapsto \log(1/|\varphi(z)|)$ apart from its poles, and the latter function vanishes on the boundary. Hence we may apply the   Green's formula and the fact   $\frac{1}{2\pi}\Delta\log|\varphi(z)|=\sum\limits_{z\in\mathbb{C};\; \varphi(z)=0}\delta_z$ on the unit disc to obtain
\begin{eqnarray*}
   \sum\limits_{k\geq 1}(a-|z_k|^2)^\beta&\asymp&
  \frac{1}{2\pi}\int_{\mathbb{D}}(a-|z|^2)^{\beta-2}\log(1/|\varphi(z)|)dA(z)+\frac{1}{2\pi}\int_{\partial D}v(z)\nabla (\log(1/|\varphi(z)|))\cdot z, 
\end{eqnarray*}
where $\cdot$ stands for the $\R^2$ inner product. By letting $a\downarrow 1^+$ we obtain
\begin{equation*}
   \sum\limits_{k\geq 1}(1-|z_k|^2)^\beta\asymp  \int_{\mathbb{D}}(1-|z|^2)^{\beta-2}\log(1/|\varphi(z)|)dA(z).
\end{equation*}
This same relation is then obtained for a general Blaschke product  by applying it first to a partial Blaschke product and using monotone convergence. Finally, the lemma follows by noting that in our setup the series   $\sum\limits_{k\geq 1}(1-|z_k|^2)^\beta$ and  $ \sum\limits_{k\geq 1}(1-|z_k|)^\beta$ converge simultaneously.
\end{proof}

\begin{coro}\label{coro:CoroThm2}
The first part of Theorem~\ref{thm2} holds, if for all $\varepsilon>0$ and all $r$ close enough to $1$, we have for our random inner function $\varphi$
\begin{equation*}
\sup\limits_{|z|=r}\mathbb{E}\left[\log\left(\frac{1}{|\varphi(z)|}\right)\right]\lesssim (1-r)^{\frac{\gamma^2}{8}-\varepsilon}.
\end{equation*}
\end{coro}

To establish the first part of Theorem~\ref{thm2}, we use a slightly weaker form of Corollary~\ref{coro:CoroThm2}. In order to state it, let us denote by $x(z)$ and $y(z)$  respectively the real and imaginary parts of $h(z)$, i.e.
\begin{equation*}
h(z)=\int_{\T}\frac{e^{i\theta}+z}{e^{i\theta}-z}d\mu(\theta)=x(z)+iy(z).
\end{equation*}
We can write them explicitly in terms of the Poisson kernel and its conjugate:
\begin{equation*}\label{eq:DefinitionsXY}
\begin{split}
x(z)=\int_{\T}P_z(e^{i\theta})d\mu(\theta),\qquad &P_z(e^{i\theta})=\frac{1-|z|^2}{|z-e^{i\theta}|^2},\\
y(z)=\int_{\T}Q_z(e^{i\theta})d\mu(\theta),\qquad &Q_z(e^{i\theta})=\frac{-2|z|\sin(\theta-\arg (z))}{|z-e^{i\theta}|^2}.
\end{split}
\end{equation*}
\begin{coro}[Upper bound for the density problem]\label{lem:DensitySufficient}
The first part of Theorem~\ref{thm2} holds if, for all $\varepsilon>0$,
\begin{equation*}\label{eq:UpperBoundEstimate}
\mathbb{E}\left[\log\left(1+\frac{4x(z)}{(x(z)-1)^2}\right)\right]\lesssim (1-|z|)^{\frac{\gamma^2}{8}-\varepsilon}\qquad\textrm{\rm for} \qquad z\in\D.
\end{equation*}
\end{coro}
This follows directly from Lemma~\ref{lemm:densityupperbound} by writing 
\begin{equation*}\label{eq:ExpansionVarphi}
\mathbb{E}\left[\log\left(\frac{1}{|\varphi(z)|}\right)\right]=\mathbb{E}\left[\log\left(\frac{|h(z)+1|}{|h(z)-1|}\right)\right]=\frac{1}{2}\mathbb{E}\left[\log\left(1+\frac{4x(z)}{(x(z)-1)^2+y(z)^2}\right)\right].
\end{equation*}

\subsection{Control of the log-singularity via perturbation}\label{subse:control}
We first establish an auxiliary result in the spirit of Lemma~\ref{lem:integrallemma}, factorizing out this time two independent Gaussian components. The idea is the same: by establishing a lower bound in the derivatives of certain directions, we can reduce the estimate with log-singularity to a relatively classical moment estimate. Now that we parametrize the Gaussian perturbation in higher dimensions, the estimation becomes considerably more complicated and we need to study level sets of convex functions with suitable derivative bounds.

The proof of the following  proposition is technical and will be postponed to Section~\ref{subse:existence}.
\begin{prop}\label{prop:RankTwoAssumption}
Consider a log-correlated Gaussian field $X$ on $\T$ with kernel
\begin{equation*}
    K(\theta,\theta')=-\log|\theta-\theta'|+g(\theta,\theta')
\end{equation*}
with $g\in W^{s,2}(\T^2)$ for some $s>1$. There exist two continuous and real functions $f_1,f_2$  on $\T$ satisfying
\begin{equation*}
\inf_{t\in [0,2\pi)}\left(|f_1(\theta)|^2+|f_2(\theta)|^2\right)>0
\end{equation*}
and a decomposition of $X$ into three independent Gaussian components,
\begin{equation*}
    X(\theta)=V_1 f_1(\theta)+V_2 f_2(\theta)+\widetilde{X}(\theta)
\end{equation*}
with $V_1,V_2$ i.i.d. standard Gaussians and independent of the residual field $\widetilde{X}$.
\end{prop}

 If we denote by $\mu$ and $\widetilde{\mu}$ the associated Gaussian multiplicative chaos to $X$ and $\widetilde{X}$, then for $z\in\D$,
\begin{equation*}
    x(z)=\int_{\T}\exp\left(\gamma V_1f_1(\theta)+\gamma V_2f_2(\theta)-\frac{\gamma^2}{2}(f_1(\theta)^2+f_2(\theta)^2)\right)P_z(\theta)d\widetilde{\mu}(\theta)
\end{equation*}
where $P_z(\theta)$ is the Poisson kernel as before. Let $y_1,y_2$ parametrize the Gaussians $V_1,V_2$ and define the random function
\begin{equation}\label{eq:definition_u00}
  u(y_1,y_2)=\int_{\T}\exp\left(\gamma y_1f_1(\theta)+\gamma y_2f_2(\theta)-\frac{\gamma^2}{2}(f_1(\theta)^2+f_2(\theta)^2)\right)P_z(\theta)d\widetilde{\mu}(\theta).
\end{equation}
Especially, by the tower law we obtain 
\begin{equation*}\label{eq:-2}
\mathbb{E}\left[\log\left(1+\frac{4x(z)}{(x(z)-1)^2}\right)\right] \;=\; \frac{1}{2\pi}\mathbb{E}\left[\int_{\R^2} \log\left(1+\frac{4u(y_1,y_2)}{(u(y_1,y_2)-1)^2}\right)e^{-\gamma^2(y_1^2+y_2^2)/2}dy_1dy_2\right].
\end{equation*}

\medskip

The key to the estimation of the above integral will be the following analytic proposition.
\begin{prop}\label{prop:LogSingularity}
Assume that $u:\mathbb{R}^2\to\mathbb{R}$ is positive, smooth, convex, and satisfies $\Delta u\geq \kappa u$ for some constant $\kappa>0$ together with $|Du|\leq Ku$ for some constant $K>0$.
\begin{enumerate}[label=(\roman*)]
    \item There is a constant $c_0<\infty$, independent of $u$, such that
    \begin{equation*}
        \int_{Q} \log\left(1+\frac{4u(y_1,y_2)}{(u(y_1,y_2)-1)^2}\right) dy_1dy_2 \leq c_0
    \end{equation*}
    for every square $Q\subset \mathbb{R}^2$ of side length $1$.
    \item The following integral is controlled by the value of $u(0,0)$:
    \begin{equation*}
    I\coloneqq\frac{1}{2\pi}\int_{\mathbb{R}^2} \left[\log\left(1+\frac{4u(y_1,y_2)}{(u(y_1,y_2)-1)^2}\right)\right]e^{-\gamma^2(y_1^2+y_2^2)/2}dy_1dy_2\leq  c_1(u(0,0)\wedge c_2).
    \end{equation*}
    The constants $c_1$ and $c_2$ depend only on $\kappa, K$ and $\gamma$.
\end{enumerate}
\end{prop}

Before proving this result we first check that the function defined via \eqref{eq:definition_u00}  satisfies the conditions of the above proposition. 
\begin{lemm}\label{le:satisfies}
The function $u$ defined in  \eqref{eq:definition_u00} almost surely satisfies the conditions of the above proposition.
\end{lemm}
\begin{proof} 
Note first  that $(y_1,y_2)\mapsto \exp(ay_1+by_2)$ is convex, and sums (or integrals of a parametrized family) of convex functions stays convex, so $u$ is convex.  Next, we may estimate
\begin{eqnarray*}
    \Delta u(y_1,y_2) &=&
  \gamma^2\int_{\T}\big( f_1(\theta)^2+f_2(\theta)^2\big)\exp\left(\gamma y_1f_1(\theta)+\gamma y_2f_2(\theta)-\frac{\gamma^2}{2}\big(f_1(\theta)^2+f_2(\theta)^2\big)\right) \\
  &&\phantom{aaaaaaaaaaaaaaaaaa}\times P_z(\theta)d\widetilde{\mu}(\theta)\\
  &\geq& A\gamma^2u(y_1,y_2),
\end{eqnarray*}
where $A=\inf_{t\in [0,2\pi)}\big(|f_1(\theta)|^2+|f_2(\theta)|^2\big)>0$. The upper bounds for the derivatives follow in a similar way, and finally the positivity of $u$ is evident.
\end{proof}

We write hereafter $y=(y_1,y_2)\in\mathbb{R}^2$ for simplicity.
\begin{lemm}\label{lem:Ahlfors}
Assume that $u:\mathbb{R}^2\to\mathbb{R}$ is positive, of class $C^2$, and satisfies $\Delta u\geq \kappa u$ with $\kappa>0$. For any ball $B=\overline{B(y_0,r)}\subset\mathbb{R}^2$, denote $M_u(B)\coloneqq\max\limits_{y\in B} u(y)$ and $m_u(B)\coloneqq\min\limits_{y\in B}u(y)$. Then
\begin{equation*}
M_u(B)\geq (1+\kappa r^2/4)m_u(B).
\end{equation*}
\end{lemm}
\begin{proof}
By translation we may assume that $y_0=0$. The assumption yields that $\Delta(u(y)-A|y|^2)\geq 0$ in $B$ for any $A\leq \kappa m_u/4$ . Thus the subharmonic function $y\mapsto u(y)-A|y|^2$ achieves its maximum at the boundary. In particular, for some $y_1\in \partial B$ we have $u(y_1)-Ar^2\geq u(0)$, which yields the claim.
\end{proof}

Intuitively, we are interested in the geometry around the level set $u(y)=1$, where the singularity is present. All non-degenerate level sets are convex Jordan curves by convexity of $u$, and the above lemma can be transformed into an upper bound on the distance between levels sets of $1$ and $1\pm t$.

\begin{proof}[Proof of Proposition~\ref{prop:LogSingularity}]
We first prove item~(i). According to Lemma~\ref{lem:Ahlfors}, our function $u$ is convex, tends to infinity and is not constant in any non-empty open set. Denote $m=\inf_{z\in\R^2} u(z)\geq 0$. The sublevel sets $\Omega_t\coloneqq\{y:u(y)<t\}$  for $t>m$ are convex, non-empty and open sets, increasing in $t$, and the level sets $S_t\coloneqq\{y:u(y)=t\}$  satisfy $\partial \Omega_t=S_t$. In turn, $S_m$ is either empty,  a point, or a closed line (segment).

Let $m<t<t_2$ and assume that $y_0\in\partial\Omega_{t}$. By convexity, for given $r>0$, we may pick an open ball $B$ of radius $r/2$ such that $B\subset\mathbb{C}\setminus\overline{\Omega}_{t}$ and $y_0\in \overline{B}$. Lemma~\ref{lem:Ahlfors} implies that $\overline{B}\cap\partial\Omega_{t_2}\neq\emptyset$ as soon as $ t(1+\kappa r^2/16)\geq t_2$. Let us denote by $A_\varepsilon\coloneqq\{y\in\mathbb{C}:d(y,A)\leq\varepsilon\}$ the $\varepsilon$-fattening of a given set $A\subset\mathbb{C}$, we thus have shown that $\partial\Omega_{t}$ is contained in $(\partial\Omega_{t_2})_\varepsilon$ if $\varepsilon\geq 4\sqrt{\kappa^{-1}(t_2/t-1)}$. Especially, if $m\leq t_1<t_2$, as $\Omega_{t_2}\setminus \Omega_{t_1}=\bigcup_{t_1\leq t<t_2}\partial \Omega_t$, we deduce that
\begin{equation}\label{eq:Flattening}
\Omega_{t_2}\setminus \Omega_{t_1} \subset (\partial\Omega_{t_2})_\varepsilon \qquad\text{with}\qquad \varepsilon=4\sqrt{\kappa^{-1}(t_2/t_1-1)}.
\end{equation}

We now estimate the size of the intersection $(\Omega_{t_2}\setminus \Omega_{t_1})\cap Q$. Assuming $\varepsilon<1/2$ we obviously have 
\begin{equation*}
(\partial\Omega_{t_2})_\varepsilon\cap Q\subset \big(\partial(\Omega_{t_2}\cap 2Q)\big)_\varepsilon,
\end{equation*}
where $2Q$ is the twice dilated cube with the same center as $Q$. The domain $\Omega_{t_2}\cap 2Q$ is convex, and its boundary has length at most $8$ (one can define a contraction by sending each point of $\partial(2Q)$ to its nearest point in $\partial(\Omega_{t_2}\cap 2Q)$ and this shortens the boundary length), and we deduce by~\eqref{eq:Flattening} that
\begin{equation*}
|(\Omega_{t_2}\setminus \Omega_{t_1})\cap Q|\;\leq\;  4\Big( 8\times \kappa^{-1/2}\sqrt{t_2/t_1-1}+ \pi  \Big(\kappa^{-1/2}\sqrt{t_2/t_1-1}\Big)^2\Big),
\end{equation*}
which implies for  $t\in (0,1/2)$ that
\begin{equation*}
|(\Omega_{1+t}\setminus \Omega_{1-t})\cap Q| \leq 80(\kappa^{-1/2}+\kappa^{-1})\sqrt{2t/(1-t)}\leq c_1\sqrt{t},
\end{equation*}
where $c_1\coloneqq 160(\kappa^{-1/2}+\kappa^{-1})$. Previously we were implicitly assuming that $m\leq 1/2$, but the analysis goes through with obvious changes when $m\in (1/2, 3/2)$, and the case $m\geq  3/2$ is not needed.

Finally, this enables us to estimate, separating the case $|u(y)-1|>1/2$ and $|u(y)-1|\leq1/2$,
\begin{equation*}
\begin{split}
&\int_{Q} \log\left(1+\frac{4u(y)}{(u(y)-1)^2}\right) d^2y\\
\leq{}&\log (1+ 6(1/2)^{-2})+\int_{(\Omega_{3/2}\setminus \Omega_{1/2})\cap Q} \log\left(1+\frac{6}{(u(y)-1)^2}\right) d^2y\\
\leq{}&\log(25)+\int_0^{1/2} \frac{c_1}{{2}}t^{-1/2}\log (1+6t^{-2})dt\eqqcolon c_0.
\end{split}
\end{equation*}

\medskip

For item~(ii), let us first assume that $u(0)\eqqcolon a< c_2$, where $c_2\leq e^{-4K}/2$ will be fixed later. Then we partition $\mathbb{R}^2$ into squares $Q_{m,n}\coloneqq(m,n)+[0,1)^2$, where $(m,n)\in\mathbb{Z}^2$, and apply the estimate of item~(i) on each of these squares separately. We invoke the knowledge  $|\nabla u|\leq Ku$, which implies in the polar coordinates $|du/dr|\leq Ku$, and $u(y)\leq e^{K|y|}a$. Especially $u(y)\leq 1/2$ inside the ball $B(0, \log(1/2a)/K)$. 

We then consider separately the integrals over the domains $B(0,\log(1/2a)/K)$ and $\mathbb{C}\setminus B(0,\log(1/2a)/K)$. In the first domain, $\log\big(1+\frac{4u(y)}{(u(y)-1)^2}\big)\lesssim u(y)$ and we bound the integral by
\begin{eqnarray*}
    &&C\frac{1}{2\pi}\int_{B(0,\log(1/2a)/K)}u(y)e^{-\gamma^2|y|^2/2}d^2y\; \lesssim \; a\int_{0}^{\infty}re^{Kr}e^{-\gamma^2r^2/2}dr\\
    &\lesssim& a\int_{0}^{\infty}re^{(1+K/\gamma) r}e^{-r^2/2}dr  \;   \lesssim \; e^{\frac{(1+K/\gamma)^2}{2}}a.
\end{eqnarray*}

When dealing with the second domain, let $l\geq \log(1/2a)/K\geq 4$ be an integer. We apply part~(i) on the square $Q_{m,n}$ with $Q_{m,n}\cap\partial B(0,l)\neq\emptyset$, and see that   the integral over $Q_{m,n}$ is at most
\begin{equation*}
    c_0 e^{-\gamma^2(l-\sqrt{2})^2/2}.
\end{equation*}
Every square $Q_{m,n}$ intersects some $\partial B(0,l)$ with an integer $l$, and for each fixed $l\neq 0$, the number of such squares is at most $250861l$. Summing up we obtain the following upper bound in the second case:
\begin{equation*}
\begin{split}
    &c_0\sum\limits_{l\geq\log(1/2a)/K,~l\in\mathbb{N}} 250861le^{-\gamma^2(l-\sqrt{2})^2/2}\\
    \leq{}& 250861c_0\int_{\log(1/2a)/K-1-\sqrt{2}}^\infty (x+\sqrt{2}) e^{-\gamma^2x^2/2}dx\\
    \leq{}& 250861c_0\int_{\log(1/2a)/K-1-\sqrt{2}}^\infty 2x e^{-\gamma^2x^2/2}dx\\
    \lesssim{}& 2\times 250861c_0e^{-\gamma^2(\log(1/2a)/K-1-\sqrt{2})^2/2}\\
    \lesssim{}& c_0 a
\end{split}
\end{equation*}
where the first  inequality holds as $x\mapsto (x+\sqrt{2}) e^{-\gamma^2x^2/2}$ is decreasing on the range of integration, and the last one e.g. if  $-\log(2a)\geq 6+2K^2/\gamma^2$. All the needed constaints are clearly satisfied for $a<c_2<1$ by choosing $c_2$ small enough depending on $K$ and $\gamma.$

In the case $a\geq c_2$, it suffices to show that $I$ is bounded by a constant independent of $a$. This easily follows from item~(i) by again summing over all $Q_{m,n}$ with $(m,n)\in\mathbb{Z}^2$.
\end{proof}

The upshot of this section is:
\begin{coro}\label{coro:FinalSufficient}
The first part of Theorem~\ref{thm2} holds if, for all $z\in\mathbb{D}$ and $\varepsilon>0$,
\begin{equation*}
    \mathbb{E}\left[\widetilde{x}(z)\wedge 1\right]\lesssim(1-|z|)^{\frac{\gamma^2}{8}-\varepsilon}
\end{equation*}
where with the notations in the beginning of this section,
\begin{equation*}
    \widetilde{x}(z)=\int_{\T}P_z(\theta)d\widetilde{\mu}(\theta).
\end{equation*}
\end{coro}
\begin{proof}
    Combine Corollary~\ref{lem:DensitySufficient}, Equation~\eqref{eq:definition_u00} and item~(ii) of Proposition~\ref{prop:LogSingularity}.
\end{proof}

\subsection{Existence of a rank two independent Gaussian summand}\label{subse:existence}

 We now prove Proposition~\ref{prop:RankTwoAssumption} by showing that one may actually choose $f_1$ and $f_2$  to be two trigonometric polynomials.  We start first by considering suitable compact perturbations of the identity operator.
\begin{lemm}\label{le:hilbert}
Assume that $T$ is compact and self-adjoint on a separable Hilbert space $H$ and $\id+T\geq 0$.
Then there are vectors $\varphi_1,\dots, \varphi_{l_0}$ with the following property: for a given vector $u\in H$ there exists $\varepsilon >0$ such that
\begin{equation*}
T-\varepsilon u\otimes u\geq 0.
\end{equation*}
if and only if
\begin{equation*}
\langle u,\varphi_j\rangle=0\qquad \textrm{for}\qquad j=1,\dots, l_0.
\end{equation*}
\end{lemm}

\begin{proof} 
Let $\{\varphi_n\}_{n=1}^\infty$ be an orthonormal basis of $H$ consisting of eiegnvectors of $T$, ordered so that the eigenvectors that correspond to eigenvalue $-1$ are listed first, and write the spectral decomposition 
\begin{equation*}
T=-\sum_{j=1}^{l_0}\varphi_j\otimes\varphi_j +\sum_{j> l_0}\lambda_j\varphi_j\otimes \varphi_j.
\end{equation*}
Here we used the compactness of the operator $T$ to deduce that $l_0$ may be taken finite, and we also see that there is $\varepsilon>0$ such that  $\lambda_j>\varepsilon-1$ for all $j>l_0.$
Assume then that $u\in H$ is normalized and satisfies the stated condition so that $u=\sum_{j> l_0}u_j\varphi_j$ with $\sum_{j> l_0}u_j^2=1$. If $x\in H$ is arbitrary, $x=\sum_{j=1}^\infty x_j\varphi_j$, we may compute
\begin{eqnarray*}
\langle x,(\id +T-\varepsilon u\otimes u)x\rangle &=&\sum_{j>l_0}(1+\lambda_j)x_j^2-\varepsilon(\sum_{j>l_0}u_jx_j)^2\\
&\geq &\varepsilon\sum_{j>l_0}x_j^2-\varepsilon (\sum_{j>l_0}x_j^2)(\sum_{j>l_0}u_j^2)=0.
\end{eqnarray*}

In turn, the necessity of the condition is seen by observing that if $b=\sum_{j=1}^\infty b_j\varphi_j$, with $b_{j_0}\neq 0$ for some $1\leq j_0\leq l_0$, we obtain $\langle \varphi_{j_0},(\id +T-\varepsilon b\otimes b)\varphi_{j_0}\rangle=-b_{j_0}^2<0$.
\end{proof}

The following example shows that it is not always enough to consider perturbations with one base element. In the context of Proposition~\ref{prop:RankTwoAssumption} this means  that choosing $f_1$ to be simply a trigonometric monomial will not always suffice.
\begin{exem}\label{co:yksi}
Let $(e_j)_{j\geq 1}$ be an orthonormal basis of $H$, and set  $a\coloneqq\sum_{j= 1}^\infty2^{-j/2}e_j$. Then $\id-a\otimes a\geq 0$. However,  for all $n\geq 1$ and $\varepsilon >0$ we have $\id-a\otimes a-\varepsilon e_n\otimes e_n\not\geq 0$.
\end{exem}

In our next auxiliary result, we generalize Lemma \ref{le:hilbert} to cover operators of form $C+A$, where the compact operator $C$ is assumed to dominate $A$ in a suitable sense (i.e. $A$ is compact with respect to $C$, see condition \eqref{eq:037} below). 
\begin{lemm}\label{le:hilbert2} 
Assume that $C,A$ are compact (symmetric) and self-adjoint operators so that both $C\geq 0$ and $C+A\geq 0$. Assume also that the kernel $M\coloneqq\kernel(C)$ is finite-dimensional, and denote by $P$ the orthogonal projection onto $M$. We write $C_1\coloneqq C+P$ and assume that 
\begin{equation}\label{eq:037}
C_1^{-1/2}AC_1^{-1/2} \qquad\textrm{extends to a compact operator on}~H.
\end{equation}
Then, there are elements $v_1,\dots, v_{l_0}\in H$ such that  if $u\in H$ satisfies $\langle u,v_j\rangle=0$ for $j=1,\dots, l_0,$ then 
\begin{equation*}
C+A- \varepsilon C_1^{1/2}u\otimes C_1^{1/2}u\geq 0
\end{equation*}
for small enough $\varepsilon >0.$
\end{lemm}
\begin{proof}
We may actually assume that $M=\{0\}$ by replacing $C$ by $C_1$ and $A$ by $A-P$. The claim then follows easily by applying the previous lemma with $T=C_1^{-1/2}AC_1^{-1/2}$.
\end{proof}

\begin{coro}\label{co:kaksi} In the situation of Lemma \ref{le:hilbert2}, let $\{e_j\}$ denote an orthonormal basis of $H$ consisting of eigenvectors of $C_1$. Then for any $E\subset \N$ with $|E|\geq l_0+1$ there is a nontrivial linear combination $b=\sum_{j\in E}a_je_j$ and $\varepsilon>0$ so that
\begin{equation*}
C+A-\varepsilon b\otimes b \geq 0.
\end{equation*}
\end{coro}
\begin{proof}
Simply note that if $u$ is a linear combination of eigenvectors of $C_1$, then $b=C_1^{1/2}u$ is also.
\end{proof}

We then apply the previous observations to our situation.
\begin{lemm}
Assume that $C$ is the covariance of the canonical field on $\T$ and let $A\in W^{2,2}(\T^2)$ be a symmetric (real-valued) function such that 
$C+A\geq 0$. We shall also denote by $C, A$  the corresponding operators on $L^2(\T)$. Then there are nontrivial finite trigonometric polynomials $f_1,f_2$ such that
\begin{equation*}\label{eq:trig1}
C+A-f_1\otimes f_1-f_2\otimes f_2\geq 0
\end{equation*}
and
\begin{equation*}\label{eq:trig2}
\inf_{\theta\in [0,2\pi)}\left(|f_1(\theta)|^2+|f_2(\theta)|^2\right)>0.
\end{equation*}
\end{lemm}
\begin{proof}
Let us denote by $(e_j)_{j\geq 0}$ the standard trigonometric basis of $L^2$, where $e_0$ is constant, and $e_{2j-1}, e_{2j}$ are the sine and cosine functions with frequency $j$ for $j\geq 1$. In this basis $C, C_1$ are the diagonal operators
\begin{eqnarray*}
C&=&{\rm diag\,}(0,1,1,1/2,1/2,1/3,1/3,\dots ), \\ C_1&=&{\rm diag\,}(1,1,1,1/2,1/2,1/3,1/3,\dots )\; \eqqcolon \;{\rm diag\,}(m),
\end{eqnarray*}
and $P=e_0\otimes e_0$ is one-dimensional. Especially, $C_1$ is the Fourier-multiplier $T_m$, where $m$ is the above defined vector, and we see immediately that $C:W^{s,2}(T)\to W^{s+1,2}(T)$ is an isomorphism. The powers  $C^\alpha$ are naturally defined. Lemma \ref{le:hilbert2} applies as soon as we verify that \eqref{eq:037} is satisfied. This follows e.g. by checking that $C_1^{-1/2}AC_1^{-1/2}$ is Hilbert-Schmidt, which is clearly equivalent to $(1-\Delta_x)^{1/4}(1-\Delta_y)^{1/4}A\in L^2(\T^2)$. Writing this in terms of the Fourier coefficients and using Cauchy-Schwarz, we see that this is implied by the fact that $A\in W^{1,2}(\T^2).$

Let us denote by $\varphi_1,\dots,\varphi_l$ an orthonormal basis of the eigenspace correponding to eigenvalue $-1$ of the operator $C_1^{-1/2}AC_1^{-1/2}$. Then the condition for a trigonometric polynomial $f$ ensuring that $C+A-\varepsilon f\otimes f\geq 0$ for some $\varepsilon >0$ can be written in the form
\begin{equation*}
\langle f, C_1^{-1/2}\varphi_j\rangle=0,\qquad j=1,\dots l,
\end{equation*}
where the computation is a priori formal. However, the element $C_1^{-1/2}\varphi_j$ belongs to $L^2(\T)$ since we check as above that $C_1^{-1}AC_1^{-1}$ extends to a bounded operator on $L^2$ (since $A\in W^{2,2}(\T^2)$), and hence we obtain
\begin{equation*}
C_1^{-1/2}\varphi_j=C_1^{-1/2}\big(-C_1^{-1/2}AC_1^{-1/2}\varphi_j\big)=-C_1^{-1}AC_1^{-1}(C_1^{1/2}\varphi_j)\in L^2(\T).
\end{equation*}
Let us denote $C_1^{-1/2}\varphi_j\eqqcolon\psi_j$ for $j=1,\dots,l$. This means that $C+A-\varepsilon f\otimes f\geq 0$ for some $\varepsilon >0$ as soon as the trigonometric polynomial $f$ satisfies
\begin{equation}\label{eq:trig4}
\langle f, \psi_j\rangle=0,\qquad j=1,\dots l.
\end{equation}

We still need to find two trigonometric polynomials $f_1$ and $f_2$, such that they both satisfy~\eqref{eq:trig4} and without common zeroes. For this purpose, we first fix $f_1$ to be any non-trivial trigonometric polynomial satisfying~\eqref{eq:trig4}. This can obviously be done by choosing the degree of $f_1$ to be at least $l+1$. The polynomial $f_1$ has only finitely many zeroes. Using a perturbative argument and induction on the number of zeroes, it clearly suffices to find $f_2$ such that it does vanish in a given point $\theta_0$ and it satisfies~\eqref{eq:trig4}. For any sequence $y=(y_n)_{n\geq 0}$ denote by $P_Ny$ the finite sequence $P_Ny\coloneqq(y_n)_{ 0\leq n\leq N}$. A trigonometric polynomial $f_2=\sum_{j=0}^Nc_je_j$ vanishes at $\theta_0$ exactly when $\langle f_2,P_Ng_0\rangle=0$, where 
\begin{equation*}
g_0\coloneqq(1,\sin(\theta_0), \cos(\theta_0), \sin(2\theta_0),\dots).
\end{equation*}
Thus the choice of $f_2$ as a trigonometric polynomial is not possible if and only if
\begin{equation*}
P_Ng_0\in {\rm span\,}\{P_N\psi_1,\dots, P_N\psi_l\}\qquad \textrm{ for all} \;\;N.
\end{equation*}
However, then an elementary argument  implies that $g_0\in {\rm span\,}\{\psi_1,\dots, \psi_l\}$, which is impossible since $g_0\notin \ell^2.$ Now the positivity of both $C+A-\varepsilon_1f_1\otimes f_1$ and $C+A-\varepsilon_2f_2\otimes f_2$ implies 
\begin{equation*}
C+A-\frac{1}{2}\varepsilon_1f_1\otimes f_1-\frac{1}{2}\varepsilon_2f_2\otimes f_2\geq 0,
\end{equation*}
and this completes the proof.
\end{proof}

\section{Proof of the upper bound on the density of zeroes}
In this section we verify the sufficient condition in Corollary~\ref{coro:FinalSufficient}, which implies the first part of Theorem~\ref{thm2}, and it also automatically implies the general form of Theorem~\ref{thm1}. Since
\begin{equation}\label{eq:s2}
	\mathbb{E}\left[\widetilde{x}(z)\wedge 1\right]\lesssim\mathbb{E}\left[\widetilde{x}(z)^{p}\right],
\end{equation}
for all $p\in(0,1)$, it suffices to establish a $p$-moment bound on $\widetilde{x}(z)$ with $p<1$. We only need the fact that $\widetilde{x}(z)$ is generated by a general log-correlated Gaussian field in the rest of the proof, so some notations are deliberately not strict (the proofs work for any Gaussian multiplicative chaos measures).

\begin{lemm}\label{lem:MomentBound2}
With the assumptions of Theorem~\ref{thm1} we have for any $p\in (0,1)$ the estimate
\begin{equation*}
\mathbb{E}\left[\widetilde{x}(z)^{p}\right] \lesssim (1-|z|)^{p(1-p)\gamma^2/2}.
\end{equation*}
\end{lemm}
\begin{proof}
We denote by  $A_{\widetilde{\mu},\varepsilon}(\theta)$ stands for the average of the Gaussian multiplicative chaos measure $d\widetilde{\mu}$ over the interval $[\theta-\varepsilon,\theta+\varepsilon]$. Let us consider our Gaussian field $\widetilde X$ on the interval $\theta\in (-1/2,1/2)$. Let  $X_1$ be an exactly scaling field on $(-1/2,1/2)$. At this point we refer to Fact~\ref{fact:ScalingPurelog} in Section~\ref{sec:chaos} for the properties of the exactly scaling field. For the chaos $\mu_1$ constructed from the exactly scaling field $X_1$ it holds that 
\begin{equation*}
\mu_1(-\varepsilon,\varepsilon)\sim 2\varepsilon^{1+\gamma^2/2}\exp\big(\gamma N(0,\log{1/\varepsilon})\big)\mu'_1(-1/2,1/2),
\end{equation*}
where $\mu'_1$ is a copy of $\mu_1$ and the normal variable $N(0,\log{1/\varepsilon})$ is independent of $\mu'_1$. This distributional equality yields immediately that $\mathbb{E}\left[(\mu_1(-\varepsilon,\varepsilon))^p\right]=c_p\varepsilon^{p(1+\gamma^2(1-p)/2)}$ (use Fact~\ref{fact:Positivemoment} in Section~\ref{sec:chaos}). By the reasoning in Remark~\ref{re:comp} we deduce the corresponding upper bound for $\widetilde{\mu}$, i.e.
\begin{equation}\label{eq:q}
\mathbb{E}\left[(\widetilde \mu(-\varepsilon,\varepsilon))^p\right]\lesssim \varepsilon^{p(1+\gamma^2(1-p)/2)}.
\end{equation}

To obtain the statement of the lemma we may assume (e.g. by Harnack) that $z=r\in (0,1)$, where $1-r=2^{-k}$ for an integer $k\geq 2$. Let us denote $q:=p(1+\gamma^2(1-p)/2).$ By a standard property of the Poisson extension (that can be read directly from the kernel) we may upper  bound $\widetilde x(r)$ by the convex combination 
\begin{equation}\label{eq:convex}
\widetilde x(r)\lesssim 2^{-k}\widetilde\mu(\T)+\sum_{j=0}^k2^{-j}A_{\widetilde \mu,2^{j-k}}(\theta).
\end{equation}
In view of this and \eqref{eq:q} we may estimate (use $q<2p$ for $p\in(0,1)$)
\begin{equation*}
\mathbb{E}\left[\widetilde x(r)^p\right] \lesssim 2^{-kp}\mathbb{E}\left[\widetilde\mu(\T)^p\right] +\sum_{j=0}^k2^{-jp}2^{(j-k)q}2^{(k-j)p}\lesssim 2^{k(p-q)}
\lesssim  (1-r)^{q-p},
\end{equation*}
which yields the claim by the definition of $q$.
\end{proof}

\begin{coro}\label{coro:FinalBound}
For all $\varepsilon>0$ and $z\in\mathbb{D}$,
\begin{equation*}
	\mathbb{E}\left[\widetilde{x}(z)\wedge 1\right]\lesssim (1-|z|)^{\frac{\gamma^2}{8}-\delta},
\end{equation*}
where $\delta>0$ is arbitrary.
\end{coro}
\begin{proof} This follows simply by choosing $p=1/2$ in  Lemma~\ref{lem:MomentBound2}. 
\end{proof}
This completes the proof for the first part of Theorem~\ref{thm2} in view of~\eqref{eq:s2}.

\medskip

We end this section by presenting an alternative  to Corollary~\ref{coro:FinalBound} that is based on the following known result:
\begin{lemm}\label{lem:MomentBound}
Let $Y$ be a log-correlated Gaussian field on the unit circle $\T$ and $\mu_{Y}$ the associated Gaussian multiplicative chaos measure with parameter $\gamma\in(0,\sqrt{2}]$. Then for all $0<p<1$ the quantity
\begin{equation*}
	\sup_{z\in \T}\mathbb{E}\left[\left(\int_{\T}\frac{1}{|z-e^{i\theta}|^{s}}d\mu_Y(\theta)\right)^{p}\right]
\end{equation*}
is finite if and only if $s<s(p)=1+\frac{\gamma^2}{2}(1-p)$.
\end{lemm}
This result is known in the physics literature on $2d$-Conformal Field Theory as the extended Seiberg bound, see \cite[Lemma~3.10]{david2016liouville} for the case on the Riemann sphere. For the reader's convenience we include below in Appendix~\ref{sec:Appendix} a self-contained proof of this fact.

\begin{proof} [Second proof of Corollary \ref{coro:FinalBound}]
The argument below does not depend on the log-correlated Gaussian field so let $\mu$ be a general Gaussian multiplicative chaos measure on $\T$. Since $|z-e^{i\theta}|^2\geq (1-|z|)^{2-s}|z-e^{i\theta}|^s$ we obtain 
\begin{equation*}
\begin{split}
	\mathbb{E}[x(z)^{p}]=\mathbb{E}\left[\left(\int_{\T}\frac{1-|z|^2}{|z-e^{i\theta}|^2}d\mu(\theta)\right)^{p}\right]\lesssim (1-|z|^2)^{(s-1)p}\mathbb{E}\left[\left(\int_{\T}\frac{1}{|z-e^{i\theta}|^{s}}d\mu(\theta)\right)^{p}\right].
\end{split}
\end{equation*}
By Lemma~\ref{lem:MomentBound}, the last expectation has a bound independent of $z$ if  $s<1+\frac{\gamma^2}{2}(1-p)$. By  choosing $p=\frac{1}{2}$ and letting $s$ approach $1+\frac{\gamma^2}{2}(1-p)$ from below we obtain, for all $\varepsilon>0$,
\begin{equation*}
	(1-|z|^2)^{(s-1)p}\leq (1-|z|^2)^{\frac{\gamma^2}{8}-\varepsilon}.
\end{equation*}
\end{proof}

\section{Proof of the lower bound on the density of zeroes}

We now prove the second part of Theorem~\ref{thm2}. Recall that  $A_{\mu,\varepsilon}(\theta)$ stands for the average of the Gaussian multiplicative chaos measure $d\mu$ over the interval $[\theta-\varepsilon,\theta+\varepsilon]$. 
We start by recalling a useful result on the multifractal spectrum of chaos measure estimating the size of the set of points $e^{i\theta}\in\T$ such that $ A_{\mu,1-r}(\theta)$ is (roughly speaking) of order $(1-r)^{o(1)}$.  According to~\cite[Theorem~3.1]{Bertacco_2023} (see also an earlier discussion of multifractal spectra in~\cite[Section 4, especially equation~(4.6)]{Rhodes_2014}) with $d=\alpha=1$  it holds that
\begin{equation*}\label{eq:multi}
{\rm dim}_{\mathcal{H}}(G_0)=1-\frac{\gamma^2}{8},\quad\text{where}\quad G_0\coloneqq \left\{ \theta\in [0,2\pi)~\big|~\lim_{r \to 1^-}\frac{\log A_{\mu, 1-r}(\theta)}{\log (1-r)}=0\right\}.
\end{equation*}
Fix $\delta>0$. Since $G_0\subset\bigcup_{k\geq 1}G_k$ with
\begin{equation*}
     G_k=\left\{ \theta\in [0,2\pi)~\big|~\frac{\log A_{\mu, 1-r}(\theta)}{\log (1-r)}\in (-\delta, \delta)\quad\text{for all}\quad r\in(1-1/k,1)\right\},
\end{equation*}
we may pick $r_0\in(1/2,1)$ so that  
\begin{equation*}\label{eq:multi2}
\text{dim}_{\mathcal{H}}(G)\geq 1-\delta/2-\frac{\gamma^2}{8}
\end{equation*}
with
\begin{equation*}\label{eq:multi3}
\quad G\coloneqq\left\{\theta\in [0,2\pi)~\big|~\frac{\log A_{\mu, 1-r}(\theta)}{\log (1-r)} \in (-\delta, \delta)\quad\text{for all}\quad r\in (r_0,1)\right\}.
\end{equation*}

We denote for given $\varepsilon>0$ the $\varepsilon$-thickening of the set $G$ by $G_\varepsilon$, i.e. $G_\varepsilon\coloneqq\{\theta~|~{\rm d\,}(\theta, G)\leq \varepsilon\}$. By standard estimates for the Hausdorff content (see~\cite[Chapter~5.5]{mattila1995}), we then have the following lower bound for the Lebesgue measure of the $\varepsilon$-thickening $G_\varepsilon$, that
\begin{equation}\label{eq:multi14}
|G_\varepsilon|\gtrsim  \varepsilon^{\gamma^2/8+\delta}.
\end{equation}
We then define the random subset $H\subset \D$ by setting
\begin{equation*}
    H\coloneqq\{re^{i\theta}~|~r\in (r_0,1),\quad \theta\in G_{1-r}\}
\end{equation*}
and claim that
\begin{equation}\label{eq:multi5}
c_0(1-|z|)^\delta \leq x(z) \leq c_1(1-|z|)^{-\delta} \quad \text{for}\quad z\in H,
\end{equation}
where the constants $c_0,c_1$ are independent of $z$ but possibly random. The lower bound is immediate, since if $z=re^{i\theta}\in H$, then there is $\theta'\in G$ with $|\theta'-\theta|\leq 1-r$ and hence 
\begin{equation*}
    x(z)\geq cA_{\mu, 2(1-r)}(\theta)\geq (c/2)A_{\mu, 1-r}(\theta')\geq \frac{c}{2}(1-r)^\delta = \frac{c}{2}(1-|z|)^\delta.
\end{equation*}
For the upper bound we proceed  as in \eqref{eq:convex} and  estimate the action of Poisson kernel at $z'\coloneqq re^{i\theta'}$ by convex combinations of the averages $A_{\mu, 2^{k}(1-r)}(\theta')$, where $k=0,1,\dots$ satisfies $2^{k}(1-r)\leq 1$. Using additonally  the fact that $x$ is positive harmonic function, Harnack's inequality allows us to estimate
\begin{equation*}
\begin{split}
x(z)&\;\; \leq \;\;  c'x(z')\;\; \lesssim \;\; c'\sum_{2^{k}(1-r)\;\leq 1}2^{-2k}A_{\mu, 2^{k}(1-r)}(\theta')\\ 
&\;\; \leq \;\;  c'\sum_{2^{k}(1-r)\leq 1-r_0}2^{-2k}A_{\mu, 2^{k}(1-r)}(\theta')+c'(1-r_0)^{-1}\mu(\T)\\&\;\;\lesssim \;\; c'(1-|z|)^{-\delta} +c'',
\end{split}
\end{equation*}
where $c''$ is a random constant. This yields \eqref{eq:multi5}.

In order to control the contribution of the imaginary part $y$ we will use a simple deterministic argument that relies on the fact that when the averages of the measure are in control, then the imaginary part cannot be too large either. Assume that $z=re^{i\theta}\in H$. Since the conjugate kernel has the upper bound $\lesssim \max (|e^{i\theta}-z|^{-1},(1-r)^{-1})$ and we have
\begin{equation*}
    \mu(\{ \theta'\,\big|\, |\theta'-\theta|\leq (1-r)2^k\}) = (1-r)2^kA_{\mu, 2^{k}(1-r)}(\theta),
\end{equation*}
it follows by considering dyadic annuli around $\theta$ such that the  inner radii are $(1-r)2^k$, with $k=0\ldots$, and the interval $(\theta-(1-r),\theta+(1-r))$ separately, that
\begin{eqnarray*}\label{eq:multi10}
|y(z)|&\lesssim& c_{r_0}\mu(\T)+ \sum_{2^k(1-r)\leq 1-r_0}A_{\mu, 2^{k}(1-r)}(\theta)\;\leq \; c+c'\log\left(\frac{1-r_0}{1-r}\right)(1-|z|)^{-\delta}\\
&\lesssim& c_2(1-|z|)^{-2\delta},
\end{eqnarray*}
where $c_2$ is a random constant.

We are now ready  to conclude the desired lower bound.
For $z\in H$ it holds that $\log|\varphi(z)|\gtrsim (1-|z|)^{-(\delta+4\delta)}$. By \eqref{eq:multi14} we obtain that
\begin{equation*}
    \int_{|z|=r}-\log|\varphi(z)|d|z|\;\gtrsim\; (1-r)^{(\frac{\gamma^2}{8}+\delta)}(1-r)^{(5\delta)}.
\end{equation*}
Since $6\delta$ can be chosen arbitrary small this implies the claim by Lemma \ref{lemm:densityupperbound}.

\appendix
\section{Proof of Lemma \ref{lem:MomentBound}}\label{sec:Appendix}

It is clearly enough to estimate the moment  for $z=1$. 
Consider the arc $I=[e^{-i\pi/8},e^{i\pi/8}]\subset\T$ and $0<p<1$. By subadditivity,
\begin{eqnarray*}
	\mathbb{E}\left(\int_{\T}\frac{1}{|1-e^{i\theta}|^{s}}d\mu_Y(\theta)\right)^{p}&\leq& \mathbb{E}\left(\int_{I}\frac{1}{|1-e^{i\theta}|^{s}}d\mu_Y(\theta)\right)^{p}+\mathbb{E}\left(\int_{\T\setminus I}\frac{1}{|1-e^{i\theta}|^{s}}d\mu_Y(\theta)\right)^{p}\\
	&\lesssim& \mathbb{E}\left(\int_{I}\frac{1}{|\theta|^{s}}d\mu_Y(\theta)\right)^{p} \; + \; C(p,\gamma),
\end{eqnarray*}
where $C(p,\gamma)<\infty$ by Fact \ref{fact:Positivemoment}.

By Kahane's convexity inequality (see Fact~\ref{fact:ScalingKahane} and the discussions thereafter), we can trade $Y$, restricted to $I$, for the exact scaling log-correlated field $\tilde{Y}$ with covariance kernel
$
 	K(\theta,\theta')=\log\frac{1}{|\theta-\theta'|}$, $\theta,\theta'\in[-\frac{\pi}{8},\frac{\pi}{8}].
$
In order to use the good scaling properties of the field $\tilde Y$, 
consider a dyadic tiling of the interval $[-\pi/8,\pi/8]$ by denoting $Q_k\coloneqq [-2^{-k}\pi/8,2^{-k}\pi/8]$, and $R_{k}\coloneqq Q_{k}\setminus Q_{k+1}$ for $k\geq 0$. We recall that the exact scaling property states  that for all $0<r<1$, in law,
\begin{equation*}
	(\tilde{Y}(x+ry))_{y\in J}=(\tilde{Y}(x+y))_{y\in J}+\sqrt{-\log r}N
\end{equation*}
for any interval $J\subset\R$ with $|J|<1$, where $N$ is a standard Gaussian independent of everything else, see Fact~\ref{fact:ScalingPurelog}. Using the relation above with $x=0$, $J=Q_{k}$ and $r=\frac{1}{2}$, with a $\varepsilon$-regularization by any scale-invariant mollifier, we get that for all $k\geq 0$,
\begin{equation}\label{eq:computation}
\begin{split}
	&\mathbb{E}\left[\left(\int_{R_{k+1}}\frac{1}{|\theta|^{s}}d\mu_{\tilde{Y}}(\theta)\right)^{p}\right]\\
	={}&\lim\limits_{\varepsilon\to 0}\mathbb{E}\left[\left(\int_{R_{k+1}}\frac{1}{|\theta|^{s}}e^{\gamma \tilde{Y}_{\varepsilon}(\theta)-\frac{\gamma^2}{2}\mathbb{E}\left[\tilde{Y}_{\varepsilon}(\theta)^2\right]}d\theta\right)^{p}\right]\\
	={}&\lim\limits_{\varepsilon\to 0}\mathbb{E}\left[\left(\int_{R_{k}}\frac{1}{|\theta/2|^{s}}e^{\gamma \tilde{Y}_{\varepsilon/2}(\theta/2)-\frac{\gamma^2}{2}\mathbb{E}\left[\tilde{Y}_{\varepsilon/2}(\theta/2)^2\right]}d\left(\frac{\theta}{2}\right)\right)^{p}\right]\\
	\stackrel{\ast}{=}{}&2^{s p-p}\lim\limits_{\varepsilon\to 0}\mathbb{E}\left[\left(\int_{R_{k}}\frac{1}{|\theta|^{s}}e^{\gamma \tilde{Y}_{\varepsilon}(\theta)+\gamma\sqrt{\log2}N-\frac{\gamma^2}{2}\mathbb{E}\left[\tilde{Y}_{\varepsilon}(\theta)^2\right]-\frac{\gamma^2}{2}\log2}d\theta\right)^{p}\right]\\
	={}&2^{s p-p(1+\gamma^2/2)+p^2\gamma^2/2}\lim\limits_{\varepsilon\to 0}\mathbb{E}\left[\left(\int_{R_{k}}\frac{1}{|\theta|^{s}}e^{\gamma \tilde{Y}_{\varepsilon}(\theta)-\frac{\gamma^2}{2}\mathbb{E}\left[\tilde{Y}_{\varepsilon}(\theta)^2\right]}d\theta\right)^{p}\right]\\
	={}&2^{s p-p(1+\gamma^2/2)+p^2\gamma^2/2}\mathbb{E}\left[\left(\int_{R_k}\frac{1}{|\theta|^{s}}d\mu_{\tilde{Y}}(\theta)\right)^{p}\right],
\end{split}
\end{equation}
where in step $(\ast)$ we used the scaling relation. 

The careful reader might has perhaps observed, that we have above actually treated only the subcritical chaos, i.e. $\gamma <\sqrt{2}$. In order to extend to the case of critical chaos, we simply use the definition via the so-called Seneta-Heyde normalization, where one defines the critical chaos  $\mu^{\gamma_c}_Y$ as the weak-$\star$ limit in probability
\begin{equation*}
    \mu^{\gamma_c}_Y=\lim\limits_{\varepsilon\to 0}\sqrt{-\log\varepsilon}\mu_{\varepsilon}^{\gamma_c},
\end{equation*}
where $\mu_{\varepsilon}^{\gamma_c}$ is the $\varepsilon$-regularized measure with parameter $\gamma_c=\sqrt{2}$,
\begin{equation*}
    d\mu^{\gamma_c}_{Y}(\theta)=\lim\limits_{\varepsilon\to 0}e^{\gamma_c Y_\varepsilon(\theta)-\frac{\gamma_c^2}{2}\mathbb{E}\left[Y_{\varepsilon}(\theta)^2\right]}d\theta.
\end{equation*}
Via this definition one immediately checks that \eqref{eq:computation} remains valid also for $\gamma=\sqrt{2}$. For other equivalent definitions and a review on the critical Gaussian chaos measures, see~\cite{powell2020critical}.

Notice that the exponent $s p-p(1+\gamma^2/2)+p^2\gamma^2/2$ is negative if and only if $s<s(p)=1+\frac{\gamma^2}{2}(1-p)$. Similar scaling relation holds between $Q_{k+1}$ and $Q_k$. When
\begin{equation*}
 	\mathbb{E}\left[\left(\int_{[-\pi/8,\pi/8]}\frac{1}{|\theta|^{s}}d\mu_{\tilde{Y}}(\theta)\right)^{p}\right]<\infty,
\end{equation*}
we have, since $Q_{k+1}\subset Q_k$,
\begin{equation*}
\begin{split}
	\mathbb{E}\left[\left(\int_{Q_1}\frac{1}{|\theta|^{s}}d\mu_{\tilde{Y}}(\theta)\right)^{p}\right]&=2^{s p-p(1+\gamma^2/2)+p^2\gamma^2/2}\mathbb{E}\left[\left(\int_{Q_0}\frac{1}{|\theta|^{s}}d\mu_{\tilde{Y}}(\theta)\right)^{p}\right]\\
	&>2^{s p-p(1+\gamma^2/2)+p^2\gamma^2/2}\mathbb{E}\left[\left(\int_{Q_1}\frac{1}{|\theta|^{s}}d\mu_{\tilde{Y}}(\theta)\right)^{p}\right],
\end{split}
\end{equation*}
in which case it is necessary to have $s<s(p)=1+\frac{\gamma^2}{2}(1-p)$. On the other hand, if this is satisfied, then  subadditivity yields for any $0<p<1$,
\begin{equation*}
	\mathbb{E}\left[\left(\int_{[-\pi/8,\pi/8]}\frac{1}{|\theta|^{s}}d\mu_{\tilde{Y}}(\theta)\right)^{p}\right]\leq\sum\limits_{k=0}^{\infty}\mathbb{E}\left[\left(\int_{R_k}\frac{1}{|\theta|^{s}}d\mu_{\tilde{Y}}(\theta)\right)^{p}\right]=C_0\sum\limits_{k=0}^{\infty}2^{(s p-p(1+\gamma^2/2)+p^2\gamma^2/2)k}
\end{equation*}
where $C_0=\mathbb{E}\left[\left(\int_{R_0}\frac{1}{|\theta|^{s}}d\mu_{\tilde{Y}}(\theta)\right)^{p}\right]<\infty$ since the Gaussian chaos measures in $1d$ (even in the critical case) have finite $p$-moments for $p<1$. \qed

\bibliographystyle{alpha}

\end{document}